\documentclass[letterpaper, oneside]{amsart}
\usepackage{layout}	

\usepackage[
]{geometry}

\usepackage[parfill]{parskip}   
\usepackage{amssymb}
\usepackage{amsthm}
\usepackage{amsmath}
\usepackage{enumitem}
\usepackage{mathtools}
\usepackage{hyperref}
\usepackage{colonequals}
\usepackage{mathrsfs}
\usepackage{color}
\usepackage{subcaption}
\usepackage{adjustbox}
\usepackage{ytableau}

\usepackage{tikz}
\usetikzlibrary{cd}
\usetikzlibrary{decorations.markings}
\usetikzlibrary{arrows.meta}
\usetikzlibrary{patterns}

\tikzset{
    triple/.style args={[#1] in [#2] in [#3]}{
        #1,preaction={preaction={draw,#3},draw,#2}
    }
}      

\usepackage{multicol}
\usepackage{multirow}
\usepackage{thmtools}
\usepackage{tensor}
\usepackage{mathabx}

\theoremstyle{plain}
\newtheorem{thm}{Theorem}[section]

\newtheorem{lem}[thm]{Lemma}
\newtheorem{prop}[thm]{Proposition}
\newtheorem{cor}[thm]{Corollary}

\theoremstyle{definition}
\newtheorem{defn}[thm]{Definition}
\newtheorem*{ack}{Acknowledgement}

\newtheorem{example}[thm]{Example}

\theoremstyle{remark}
\newtheorem*{rmk}{Remark}

\numberwithin{equation}{section}

\newcommand{\bC}{\mathbb{C}}

\newcommand{\bN}{\mathbb{N}}

\newcommand{\Z}{\mathbb{Z}}
\newcommand{\I}{\mathbb{I}}

\newcommand{\bB}{\mathbb{B}}

\newcommand{\cB}{\mathcal{B}}
\newcommand{\cH}{\mathcal{H}}
\newcommand{\D}{\mathcal{D}}
\newcommand{\X}{\mathcal{X}}

\newcommand{\cA}{\mathcal{A}}
\newcommand{\cO}{\mathcal{O}}
\newcommand{\cC}{\mathcal{C}}

\newcommand{\cP}{\mathcal{P}}

\newcommand{\cX}{\mathcal{X}}
\newcommand{\cN}{\mathcal{N}}

\newcommand{\fb}{\mathfrak{b}}
\newcommand{\fh}{\mathfrak{h}}
\newcommand{\fg}{\mathfrak{g}}

\newcommand{\extW}{\widetilde{W}}
\newcommand{\affW}{{W}_a}
\newcommand{\twc}{{\underline{c}}}
\newcommand{\jr}{{\mathcal{J}}}
\newcommand{\Rmat}{\mathcal{R}}
\newcommand{\md}{\textup{mod}}

\newcommand{\tc}[1]{\check{#1}}

\newcommand{\br}[1]{\left\langle{#1}\right\rangle}

\newcommand{\qlbar}{{\overline{\mathbb{Q}_\ell}}}
\newcommand{\floor}[1]{\left\lfloor#1\right\rfloor}

\DeclareMathOperator{\fgr}{\mathcal{Q}}

\DeclareMathOperator{\tr}{\textup{tr}}

\DeclareMathOperator{\Ad}{\textup{Ad}}

\DeclareMathOperator{\Aut}{\textup{Aut}}

\DeclareMathOperator{\Irr}{\textup{Irr}}

\DeclareMathOperator{\Lie}{\textup{Lie}}

\title{On the affine Sch{\"u}tzenberger involution}
\author{Dongkwan Kim}
\address{School of Mathematics\\
  University of Minnesota Twin Cities\\
  Minneapolis, MN 55455\\
  U.S.A.}
\email{kim00657@umn.edu}
\date{\today}							% Activate to display a given date or no date

\begin{document}
\begin{abstract} 
We consider an involution on the affine Weyl group of type $A$ induced from the nontrivial automorphism on the (finite) Dynkin diagram. We prove that the number of left cells fixed by this involution in each two-sided cell  is given by a certain Green polynomial of type $A$ evaluated at -1.
\end{abstract}

\setcounter{tocdepth}{1}
\maketitle

\renewcommand\contentsname{}
\tableofcontents

\section{Introduction}

This paper is motivated by a former result of Lusztig \cite{lus83:left}. Let $W$ be the Weyl group of type $A_{n-1}$ and $\tc{W}$ be the centralizer of the longest element in $W$. Then $\tc{W}$ is naturally the Weyl group of type $C_m$ where $m=\floor{n/2}$. If $L: \tc{W} \rightarrow \bN$ is the restriction of the usual length function $l: W \rightarrow \bN$, %then $(\tc{W}, L)$ is said to be in the quasisplit case in the sense of Lusztig \cite{lus14:hecke}. Then, 
then \cite{lus83:left} gives a description of left cells in $\tc{W}$ with the weight function $L$, and showed that each left cell carries an irreducible representation of $\tc{W}$.

One of our goals in this paper is to extend his method to an affine setting. Let $\affW$ be the affine Weyl group of type $\tilde{A}_{n-1}$ and let $\omega$ be the involution of $\affW$ which corresponds to the nontrivial automorphism of the (finite) Dynkin diagram of type $A$. Then, the main theorem in this paper states that the number of left cells fixed by $\omega$ in each two-sided cell of $\affW$ is given by a certain Green polynomial of type $A$, originally defined by Green \cite{gre55}, evaluated at -1. 

This paper is considered as a companion of \cite{cfkly}. There exists a bijection defined by Shi \cite{shi86, shi91} between left cells of $\affW$ and row-standard Young tableaux, called the generalized Robinson-Schensted correspondence. Under this bijection, the involution $\omega$ corresponds to an affine analogue of the usual Sch{\"u}tzenberger involution. The combinatorics of this involution is extensively studied in \cite{cfkly}, and the main theorem in this paper also follows from a more general result therein. Here, instead we explain this involution $\omega$ in a representation-theoretic view, and also provide another proof of our main theorem in terms of representation theory.

This paper is organized as follows: in Section \ref{section:notation} we cover basic notations and definitions used in this paper; in Section \ref{section:involution} we define and describe the involution $\omega$, called the affine Sch{\"u}tzenberger involution; in Section \ref{section:mainthm} we state the main theorem of this paper and remark some related facts; in Section \ref{section:proof} we prove the main theorem modulo some combinatorial reduction; in Section \ref{sec:comb} we introduce a combinatorial model associated with our main objects and complete the proof of the main theorem.

%Even though it is important to understand combinatorial pictures associated with our argument, we believe that Section 6 is not essential in order to prove the main theorem. In this paper, we heavily rely on the theory developed by Bezrukavnikov, and it is likely that there is more direct proof than ours, avoiding combinatorial argument in Section 6, by elaborating his method. However, it also implies that our proof lacks such combinatorial picture; it is desirable to find a purely combinatorial proof of the main theorem.

\begin{ack} The author thanks Roman Bezrukavnikov for kindly explaining his theory to him; it would have been impossible to write this paper without his help. The author is indebted to Michael Chmutov, Gabriel Frieden, and Joel Brewster Lewis for discussions about combinatorial descriptions of this subject, including contents in Section 6. He also thanks Sam Hopkins and Thomas McConville for making such discussions possible. Finally, he is grateful to George Lusztig and Devra Johnson for their helpful comments.
\end{ack}

\section{Notations and definitions} \label{section:notation}
\subsection{Basic notations} For a set $X$, we define $|X|$ to be the cardinal of $X$. If there is a map $f\colon X \rightarrow X$, we denote by $X^f$ the set of elements in $X$ fixed by $f$. Similarly, if there exists a group $G$ acting on $X$, then we denote by $X^G$ the set of elements in $X$ fixed by the action of $G$.

For a variety $X$, we denote by $H^i(X)=H^i(X, \qlbar)$ its $i$-th $\ell$-adic cohomology group, and define $H^*(X) \colonequals \bigoplus_{i\in \Z} (-1)^iH^i(X)$ to be their alternating sum (as a virtual $\qlbar$-vector space). 

For an abelian category $\cC$, we write $D^b(\cC)$ to be its bounded derived category, and $K(\cC)=K(\cC)_\bC$ to be the complexified Grothendieck group of $\cC$. Likewise, for a variety $X$, we write $K(X)=K(X)_\bC$ to be the complexified $K$-theory of $X$, i.e. the complexified Grothendieck group of the category of coherent sheaves on $X$. If there is an action of a group $G$ on $X$, then we denote by $K^G(X)=K^G(X)_\bC$ the complexified $G$-equivariant $K$-theory of $X$.

\subsection{General setup}
We fix $n \in \Z_{>0}$ throughout this paper.
% and let $m \colonequals \floor{n/2}$. 
Let $G\colonequals GL_{n}$ be the general linear group of rank $n$ defined over $\bC$, $B\subset G$ be the Borel subgroup consisting of upper-triangular matrices, and $T\subset B$ be the maximal torus consisting of diagonal matrices. Also let $\fg \colonequals \Lie G$, $\fb\colonequals \Lie B$, and $\fh \colonequals \Lie T$ be corresponding Lie algebras. Define $\cB\colonequals G/B$ to be the flag variety of $G$. For a nilpotent element $N \in \fg$, let $\cB_N$ be the Springer fiber of $N$, defined by $\cB_N \colonequals \{ gB \in \cB \mid \Ad(g)^{-1}(N) \in \fb\}$. 

\subsection{Weyl groups of $G$}
Let $W$ and $\extW$ be the Weyl group and the extended affine Weyl group of $G$, respectively. They are defined as follows.
\begin{gather*} W\colonequals N(G, T)/T, \qquad \extW \colonequals N(G_{\bC((t))}, T_{\bC((t))}) / T_{\bC[[t]]}.
\end{gather*}
Here, $N(X, Y)$ denotes the normalizer of $Y$ in $X$, and $G_R$ (resp. $T_R$) is the base change of $G$ (resp. $T$) from $\bC$ to $R$. Also, we define $\affW \subset \extW$ to be the subgroup generated by elements in $N(G_{\bC((t))}, T_{\bC((t))})$ whose determinant is contained in $\bC[[t]]^\times$, called the affine Weyl group of $G$. Then, $W$ is naturally a subgroup of $\affW$.

We choose $\{s_1, \ldots, s_{n-1}\} \subset W$ such that $s_i$ corresponds to swapping $i$-th and $(i+1)$-th entries of diagonal matrices. Also we let $s_0, \tau \in \extW$ be the images of
$$
\begin{pmatrix}0 & 0 & 0&\cdots&0 & 0 & t
\\0 & 1 &0& \cdots &0& 0 & 0
\\0 & 0 &1& \cdots &0& 0 & 0
\\\cdots & \cdots & \cdots & \cdots & \cdots& \cdots & \cdots
\\0 & 0 &0& \cdots &1& 0 & 0
\\0 & 0 &0& \cdots &0& 1 & 0
\\t^{-1} & 0 & 0&\cdots & 0 & 0&0
\end{pmatrix} \textup{ and } \begin{pmatrix}0 & 0 & 0&\cdots&0 & 0 & t
\\1 & 0 &0& \cdots &0& 0 & 0
\\0 & 1 &0& \cdots &0& 0 & 0
\\\cdots & \cdots & \cdots & \cdots & \cdots& \cdots & \cdots
\\0 & 0 &0& \cdots &0& 0 & 0
\\0 & 0 &0& \cdots &1& 0 & 0
\\0 & 0 & 0&\cdots & 0 & 1&0
\end{pmatrix},$$
respectively. Then, 
\begin{enumerate}[label=$\bullet$]
\item $(W,\{s_1, \ldots, s_{n-1}\})$ and $(W_a, \{s_0, s_1, \ldots, s_{n-1}\})$ are Coxeter groups, 
\item $\tau s_i \tau^{-1} = s_{i+1}$ for $0\leq i \leq n-2$ and $\tau s_{n-1} \tau^{-1} = s_0$, and
\item $\extW \simeq \affW \rtimes \br{\tau}, \br{\tau} \simeq \Z$.
\end{enumerate}
We define $S\colonequals \{s_0, s_1, \ldots, s_{n-1}\}$ to be the set of simple reflections of $\affW$.

\subsection{Involution $\omega$} Define $J \in G$ to be the matrix whose anti-diagonal entries are 1 and other entries are 0. We define an involutive automorphism $\omega$ on $G$ by
$$\omega: G \mapsto G : g \mapsto J(^tg^{-1})J^{-1}.$$
We abuse notation and write $\omega$ for an involution on any object which is naturally induced from the above automorphism. Clearly, it induces an involution on $\fg$ defined by $\omega(X) = -\Ad(J)(^tX)$. Also, it induces involutions on $W, \affW,$ and $\extW$, respectively. Indeed, direct calculation shows that
$$\omega(s_i) = s_{n-i} \textup{ for } 1 \leq i \leq n-1,\qquad \omega(s_0)=s_0,\qquad \omega(\tau) = \tau^{-1}.$$
On the other hand, since $\omega$ fixes $B$, it defines an action on the flag variety of $G$. If $N \in \fg$ is a nilpotent element fixed by $\omega$, then $\omega$ acts on its Springer fiber $\cB_N$ and thus acts on the cohomology and the $K$-theory of $\cB_N$ as well.

%\red{
%\subsection{Affine Weyl group of type $C$}\label{sec:weylc} Recall that $m = \floor{\frac{n}{2}}$. We define
%\begin{gather*}
%\tc{s}_0\colonequals s_0, \qquad \tc{s}_1 \colonequals s_1s_{n-1}, \qquad \tc{s}_2\colonequals s_2 s_{n-2}, \qquad \ldots, 
%\\\tc{s}_m \colonequals s_m \textup{ (resp. } \tc{s}_m \colonequals s_m s_{m+1}s_m)\quad \textup{ if $n$ is even (resp. odd)}.\end{gather*}
%Define $\tc{S}\colonequals \{\tc{s}_0, \tc{s}_1, \ldots, \tc{s}_m\}$ and $\tc{\affW}$ to be the subgroup of $\affW$ generated by $\tc{S}$. Then $(\tc{\affW}, \tc{S})$ is the affine Weyl group of type $C_m$, and also $(\affW)^\omega = \tc{\affW}$. We define $L: \tc{\affW} \rightarrow \bN$ to be the restriction of the usual length function $l: \affW \rightarrow \bN$, which is again a well-defined weight function in the sense of \cite{lus14:hecke}. Clearly we have
%$$L(\tc{s}_0) = 1, \qquad L(\tc{s}_i) = 2 \textup{ for } 1\leq i \leq m-1, \qquad L(\tc{s}_m) = 1 \textup{ (resp. 3) if $n$ is even (resp. odd)}.$$
%We often write $(\tc{\affW}, L)$ instead of $\tc{\affW}$ to specify this weight function.
%}

\subsection{Kazhdan-Lusztig cells} \label{sec:cells}
 We usually use the symbol $\twc$ (resp. $\Gamma$, $\Gamma^{-1}$) to denote a two-sided (resp. left, right) cell of $\extW$. (The notion of Kazhdan-Lusztig cells is first defined in \cite{kalu79} for Coxeter groups, and is generalized to extended affine Weyl groups in \cite{lus89:cell}.) Then each two-sided (resp. left, right) cell of $\affW$ is of the form $\twc \cap \affW$ (resp. $\Gamma \cap \affW$, $\Gamma^{-1} \cap \affW$), and this gives a bijection between two-sided (resp. left, right) cells of $\affW$ and $\extW$. Note that each two-sided (resp. left, right) cell of $\extW$ is stable under multiplication (resp. left multiplication, right multiplication) by $\tau \in \extW$.
% 
% \red{ Likewise, we use $\tc{\twc}$ (resp. $\tc{\Gamma}, \tc{\Gamma}^{-1}$) to denote a two-sided (resp. left, right) cell of $(\tc{\affW},L)$. (See \cite{lus14:hecke} for the definition of Kazhdan-Lusztig cells for ``unequal parameter" case.)}

In \cite{lus89:cell}, a canonical bijection between two-sided cells of $\extW$ and nilpotent orbits in $\fg$ is constructed. (Here, we identify $\fg$ with its Langlands dual.) We write $\twc_\lambda$ to be the two-sided cell that corresponds to the nilpotent orbit in $\fg$ of Jordan type $\lambda \vdash n$ under this bijection. Then $a(\twc_\lambda)$ is equal to the dimension of the Springer fiber corresponding to this nilpotent orbit, where $a$ is Lusztig's $a$-function defined in \cite{lus85:cell}.

\subsection{Partitions, tableaux, and tabloids} For a partition $\lambda$, we write $\lambda = (\lambda_1, \ldots, \lambda_r)$ for $\lambda_1 \geq \lambda_2 \geq \cdots\geq  \lambda_r >0$ or $\lambda=(1^{m_1}2^{m_2} \cdots)$ to describe its parts. Let $l(\lambda)$ be the length of $\lambda$, and put $\lambda_i=0$ if $i > l(\lambda)$. Define $|\lambda| \colonequals \sum_{i\geq 1} \lambda_i = \sum_{i\geq 1}im_i$. If $|\lambda|=k$, we also write $\lambda \vdash k$. We say $\lambda$ is strict if $\lambda_1 > \lambda_2> \cdots > \lambda_r >0$, or equivalently each $m_i$ is either 0 or 1. For another partition $\mu$, we write $\lambda \cup \mu$ to be the partition of $|\lambda|+|\mu|$ whose parts are the union (as a multiset) of parts of $\lambda$ and $\mu$.

%We often identify partitions with the corresponding Young diagram with respect to English notations, thus the length of each row is $\lambda_1, \lambda_2, \ldots,$ respectively.
\ytableausetup{smalltableaux}
For a partition $\lambda$, define $SYT(\lambda)$ to be the set of standard Young tableaux of shape $\lambda$ with entries $1, 2, \ldots, |\lambda|$. Similarly, we define $RSYT(\lambda)$ to be the set of row-standard Young tableaux of shape $\lambda$ with entries $1, 2, \ldots, |\lambda|$, which is defined by dropping the condition from $SYT(\lambda)$ that each column is strictly increasing. If $\lambda$ is a finite sequence of positive integers which is not necessarily a partition, we still write $RSYT(\lambda)$ to denote the set of row-standard Young tabloids of shape $\lambda$. For a tabloid $T$, we write $T=(T^{(1)}, \ldots, T^{(r)})$ where $T^{(i)}$ is the $i$-th part of $T$. For example, if $T\in RSYT((3,4))$ is $\begin{ytableau}1&3&4\\2&5&6&7\end{ytableau}$ with respect to the English notation, then $T^{(1)}=(1,3,4)$ and $T^{(2)}=(2,5,6,7)$.

% We use the symbol $\twc$ (resp. $\Gamma, \Gamma^{-1}$) to denote a two-sided (resp. left, right) cell of $\extW$, and similarly $\tc{\twc}$ (resp. $\tc{\Gamma}, \tc{\Gamma}^{-1}$) for $\tc{\affW}$. By the work of Lusztig and Shi \cite{lus83:padic}, \cite{shi86}, \cite{shi91}, two-sided cells in $\extW$ are parametrized by partitions under the generalized Robinson-Schensted algorithm and left cells in each two-sided cell are parametrized by row-standard Young tableaux of the corresponding shape. For a partition $\lambda \vdash n$ and a row-standard Young tableau $T$ of shape $\lambda$, we write $\twc_\lambda$ and $\Gamma_T$ to denote such cells.

\section{The affine Sch{\"u}tzenberger involution} \label{section:involution}
The involution $\omega$ on $G$ induces an involution on $\extW$, which also permutes Kazhdan-Lusztig cells of $\extW$.
\begin{defn} \emph{The affine Sch{\"u}tzenberger involution} is the involution induced by $\omega$ on the set of left cells of $\extW$, again denoted by $\omega$.
\end{defn}
If we restrict $\omega$ to the Coxeter group $(W, S-\{s_0\})$, then it corresponds to the nontrivial involutive automorphism on the Dynkin diagram of type $A$, which is the same as conjugation by the longest element of $W$. This involution clearly permutes the left cells of $W$, and it is equivalent to the usual Sch{\"u}tzenberger involution on standard Young tableaux under the Robinson-Schensted correspondence. This is why we call $\omega$ the affine Sch{\"u}tzenberger involution.

Note that the usual Sch{\"u}tzenberger involution preserves the shape of each standard Young tableau. It means that the corresponding involution on $W$ stabilizes each two-sided cell of $W$. The same is true for $\omega$, as the following lemma shows.

\begin{lem}\label{lem:2sstab} Suppose that $\twc$ is a two-sided cell of $\extW$. Then $\omega(\twc)=\twc$.
\end{lem}
\begin{proof}
It is clear that $\omega(\twc)$ is also a two-sided cell of $\extW$. Recall that $\twc$ is stable under multiplication by $\tau$, thus in particular under its conjugation. Thus by \cite[Theorem 4.8(d)]{lus89:cell}, $\twc$ intersects nontrivially with $W$. As we observed already that $\omega$ stabilizes each two-sided cell of $W$, the result follows.
\end{proof}

Therefore, it is possible to restrict $\omega$ to each two-sided cell of $\extW$. We are interested in the number of left cells in each two-sided cell that are fixed by $\omega$.

\section{Main theorem and some remarks} \label{section:mainthm}
\subsection{Statement of the main theorem}
For $k \in \bN$, let $\rho_2(k)$ be the partition $(2, 2, \ldots, 2) \vdash k$ (resp. $(2, 2, \ldots, 2, 1)\vdash k$) if $k$ is even (resp. odd). The main result of this paper is as follows.

\begin{thm}[Main theorem] \label{thm:main} Suppose that a two-sided cell $\twc_\lambda\subset \extW$ corresponds to the partition $\lambda \vdash n$. Then the number of left cells in $\twc_\lambda$ that are fixed by $\omega$ is given by $\fgr^\lambda_{\rho_2(n)}(-1)$, where $\fgr^\lambda_{\rho_2(n)}(t)$ is the Green polynomial (for type $A$) originally defined in \textup{\cite{gre55}}.
\end{thm}
It is equivalent to \cite[Theorem 4.2]{cfkly}, which is proved by a more general combinatorial result. In this paper we give another proof of this theorem using representation theory.

\subsection{Affine Weyl groups of type $C$}
Let $\tc{\affW}$ be the set of elements in $\affW$ fixed by $\omega$ and $L$ be the restriction of $l: \affW \rightarrow \bN$ to $\tc{\affW}$. Then $\tc{\affW}$ is the affine Weyl group of type $C$ whose simple reflections are given by
\begin{align*}
&s_0, s_1s_{n-1}, \ldots, s_{\frac{n-3}{2}}s_{\frac{n+3}{2}}, s_{\frac{n-1}{2}}s_{\frac{n+1}{2}}s_{\frac{n-1}{2}} &&\textup{ when $n$ is odd, and}
\\&s_0, s_1s_{n-1}, \ldots, s_{\frac{n-2}{2}}s_{\frac{n+2}{2}}, s_{\frac{n}{2}} &&\textup{ when $n$ is even}.
\end{align*}
The pair $(\tc{\affW}, L)$ is said to be in the quasisplit case in the sense of \cite{lus14:hecke}. There exists a strong connection between cells in $\affW$ and $(\tc{\affW}, L)$ as the following lemma shows. 
\begin{lem} Suppose that $\twc$ (resp. $\Gamma$) is a two-sided (resp. left) cell of $\extW$. Thus $\twc \cap \affW$ (resp. $\Gamma \cap \affW$) is a two-sided (resp. left) cell of $\affW$.
\begin{enumerate}[label=\textup{(\alph*)}]
\item $\Gamma \cap \tc{\affW}$ is nonempty if and only if $\Gamma$ is stable under $\omega$.
\item If $\Gamma \cap \tc{\affW}$ is nonempty, then it is also a left cell of $(\tc{\affW},L)$.
\item $\twc$ is always $\omega$-stable, and $\twc\cap \tc{\affW}$ is a (nonempty) union of two-sided cells of $(\tc{\affW},L)$.
\end{enumerate}
\end{lem}
\begin{proof}
For (a), one direction is clear since $(\affW)^\omega = \tc{\affW}$. For the other direction, first note that there exists $\D\subset \affW$ (the set of distinguished involutions), such that $\omega(\D) = \D$ and $\D\cap \Gamma$ consists of only one element for each $\Gamma$. Therefore, if $\omega(\Gamma) = \Gamma$ then the unique element $\D\cap \Gamma$ is fixed by $\omega$, and thus $\tc{\affW} \cap \Gamma$ is nonempty.

% Likewise, for each $\twc$ there exists a (unique) canonical left cell in $\twc$ (cf. \cite{luxi88}), say $\Gamma$. If $\omega(\twc)=\twc$, then $\omega$ also fixes $\Gamma$, thus $\twc \cap \tc{\affW}$ is also nonempty by previous observation.

For (b), we rely on the results of \cite{lus14:hecke}. Since $\affW$ is tame (see \cite[1.11, 1.15]{lus14:hecke}), it is bounded in the sense of \cite[13.2]{lus14:hecke} (also see \cite[13.4]{lus14:hecke} and \cite[Theorem 7.2]{lus85:cell} for its proof). Therefore, the argument in \cite[Chapter 16]{lus14:hecke} is applicable and (b) follows from \cite[Lemma 16.21]{lus14:hecke}.

For (c), the first part is exactly Lemma \ref{lem:2sstab}.
%note that $\twc$ intersects nontrivially with $W$. (This follows from \cite[Theorem 4.8(d)]{lus89:cell} and the fact that any two-sided cell of $\extW$ is stable under conjugation by $\tau$.) Also, the restriction of $\omega$ to $W$ is the same as conjugation by the longest element of $W$, which clearly stabilizes each two-sided cell of $W$. Thus $\twc$ is also $\omega$-stable. 
To show that $\twc\cap \tc{\affW}$ is nonempty, we consider the canonical left cell $\Gamma \subset \twc$ defined in \cite{luxi88}. Then clearly $\Gamma$ is $\omega$-stable, thus $\twc \cap \tc{\affW} \supset \Gamma \cap \tc{\affW}$ is nonempty by (a). Finally, $\twc \cap \tc{\affW}$ is a union of two-sided cells of $(\tc{\affW}, L)$ by \cite[Lemma 16.20(b)]{lus14:hecke} (and its right analogue).
\end{proof}
\begin{rmk} In general, $\twc \cap \tc{\affW}$ is not a single two-sided cell. For example, if $n=4$ and $m=2$, there are 6 two-sided cells in $(\tc{\affW}, L)$, but $\extW$ has only 5 two-sided cells. Indeed, the second highest two-sided cell of $\extW$ (which contains $S$) splits into two two-sided cells $\{s_0\}$, $\{s_2\}$ of $\tc{\affW}$. (ref. \cite[p.40]{gui08})
\end{rmk}

Now the following corollary is an immediate consequence.
\begin{cor} \label{cor:equiv} Let $\twc$ be a two-sided cell of $\extW$. Then the number of left cells of $(\tc{\affW},L)$ in $\twc \cap \tc{\affW}$ is equal to the number of left cells of $\extW$ in $\twc$ fixed by the involution $\omega$.
\end{cor}

\subsection{Relation to Domino tableaux and Springer theory}
Let $\tc{G}$ be $SO_{n}$ (resp. $Sp_{n}$) over $\bC$ if $n$ is odd (resp. even) and $\tc{\fg}$ be its Lie algebra. Regard $\tc{W} \colonequals W^\omega$ as the Weyl group of $\tc{G}$ in a natural way. For a nilpotent element $\tc{N} \in \tc{\fg}$, let  $A_{\tc{N}}$ be the component group of the stabilizer of $\tc{N}$ in $\tc{G}$. Let $\tc{\cB}$ be the flag variety of $\tc{G}$ and $\tc{\cB}_{\tc{N}}$ be the Springer fiber of $\tc{N}$.

There exists a canonical bijection between two-sided cells in $\tc{W}$ (with equal parameters) and special nilpotent orbits in $\tc{\fg}$. % (or equivalently in the Lie algebra of the Langlands dual of $\tc{G}$).
Pick a two-sided cell $\tc{\twc} \subset \tc{W}$ and let $\tc{N}\in \tc{\fg}$ be the nilpotent element in the corresponding special nilpotent orbit. Also let $\lambda$ be the Jordan type of $\tc{N}$. Then it follows from the results of Barbasch-Vogan \cite{bv82} and Garfinkle \cite{gar90, gar92, gar93} that the number of left cells in $\tc{\twc}$ is equal to that of standard domino tableaux of shape $\lambda$. It is also the same as the number of $A_{\tc{N}}$-orbits in the set of irreducible components of $\tc{\cB}_{\tc{N}}$, see \cite{mcg99, mcg00}.

%\cite[Conjecture 3.15]{lus87:leading}, proved by Lusztig (unpublished) and \cite{bez04, beos04, bfo09}, says that the number of left cells in $\tc{\twc}$ is the same as the number of $A_{\tc{N}}$-orbits in the set of irreducible components of $\tc{\cB}_{\tc{N}}$, which is also the same as the number of standard domino tableaux of shape $\lambda$ (see \cite{pie04} and the references therein).

This statement has an ``unequal'' analogue as follows. If we restrict $L: \tc{\affW} \rightarrow \bN$ to $\tc{W}$, then again $(\tc{W}, L|_{\tc{W}})$ is in the quasisplit case in the sense of \cite{lus14:hecke}. Thus similarly to the affine case above, the left cells of $(\tc{W}, L|_{\tc{W}})$ are precisely an intersection of $\tc{W}$ and some left cell of $W$ fixed by $\omega$ (see \cite{lus83:left}). For a partition $\lambda \vdash n$, let $\twc_\lambda \subset W$ be the two-sided cell of $W$ parametrized by $\lambda$. Then the number of left cells of $(\tc{W}, L|_{\tc{W}})$ contained in $\twc_\lambda \cap \tc{W}$ is equal to the number of standard domino tableaux of shape $\lambda$. In particular, if $\lambda$ is the Jordan type of a nilpotent element $\tc{N} \in \tc{\fg}$ (not necessarily special), then it is again the same as the number of $A_{\tc{N}}$-orbits in the set of irreducible components of $\tc{\cB}_{\tc{N}}$. 

The statement above also has an ``affine'' analogue. For simplicity, let us assume that $n$ is odd, thus $\tc{G} = SO_{n}$ is of type $B$. Then $\tc{\affW}$ is naturally the affine Weyl group of the Langlands dual of $\tc{G}$. There exists a canonical bijection between nilpotent orbits in $\tc{\fg}$ and two-sided cells of $\tc{\affW}$ (with equal parameters) defined in \cite{lus89:cell}. Pick a two-sided cell $\tc{\twc} \subset \tc{\affW}$ and let $\tc{N} \in \tc{\fg}$ be a nilpotent element in the corresponding nilpotent orbit. Then, a weaker version of \cite[Conjecture 10.5]{lus89:cell}, proved in \cite{bez04, beos04, bfo09}, implies that the number of left cells in $\tc{\twc}$ is given by the dimension of $(H^*(\tc{\cB}_{\tc{N}}))^{A_{\tc{N}}}$.

The main theorem in this paper should be considered as an ``affine unequal" analogue of the first statement. 
Again let $G$ be $SO_n$ or $Sp_n$ depending on the parity of $n$, and let $\tc{N} \in \tc{\fg}$ be a nilpotent element of Jordan type $\lambda \vdash n$. Let $\twc_\lambda \subset \extW$ be the two-sided cell of $\extW$ parametrized by $\lambda$. Then the result of \cite{kim:total} together with Theorem \ref{thm:main} implies that the number of left cells of $(\tc{\affW}, L)$ contained in $\twc_\lambda\cap \tc{\affW}$ is equal to the Euler characteristic of $\tc{\cB}_{\tc{N}}$, i.e. the dimension of $H^*(\tc{\cB}_{\tc{N}})$.

\section{Proof of the main theorem} \label{section:proof}
\subsection{Reduction to strict partitions}
First, we claim that in order to prove the main theorem it suffices only to consider the case when a two-sided cell $\twc \subset \extW$ corresponds to a strict partition. This follows from two propositions below.

\begin{prop} \label{prop:str} Suppose that $\lambda=(\lambda_1, \ldots, \lambda_r)\vdash n-2k$ is a partition for some $k\geq 1$. Let $\varphi(\lambda\cup(k,k))$ be the number of left cells in $\twc_{\lambda\cup (k,k)}$ fixed by $\omega$. We define $\varphi(\lambda)$ similarly (by replacing $n$ with $n-2k$, etc.) Then, for $m = \floor{n/2}$ we have
$$\varphi(\lambda\cup(k,k))=\binom{m}{k} 2^k\varphi(\lambda).$$
\end{prop}
Its proof relies on combinatorics, which we postpone until Section \ref{sec:comb}. We refer readers to \cite{cfkly} for detailed combinatorial descriptions of the affine Sch{\"u}tzenberger involution.

\begin{prop} \label{prop:grind}Suppose that $\lambda=(\lambda_1, \ldots, \lambda_r) \vdash n-2k$ is a partition for some $k\geq 1$. Then, 
$$\fgr^{\lambda\cup(k,k)}_{\rho_2(n)}(-1)=\binom{m}{k} 2^k\fgr^{\lambda}_{\rho_2(n-2k)}(-1)$$
where $m = \floor{n/2}$ and $\fgr^{\lambda\cup(k,k)}_{\rho_2(n)}(t)$, $\fgr^{\lambda}_{\rho_2(n-2k)}(t)$ are the corresponding Green polynomials.
\end{prop}
\begin{proof} Let $Q'_\lambda(t)$ be the modified Hall-Littlewood $Q'$-function. (See \cite{kir98} for its definition and properties.)  Note that $Q'_\lambda(t) = t^{b(\lambda)} \sum_{\rho\vdash |\lambda|} z_\rho^{-1}\fgr^\lambda_\rho(t^{-1})p_\rho$ where $b(\lambda) = \sum_{i\geq 1} (i-1)\lambda_i$, $z_\rho = \prod_{i\geq 1} i^{m_i} m_i!$ for $\rho = (1^{m_1}2^{m_2}\cdots)$, and $p_\rho$ is a power symmetric function. Thus we have
$$\br{Q'_{\lambda \cup (k,k)}(-1), p_{\rho_2(n)}} = (-1)^{b(\lambda \cup (k,k))}\fgr^{\lambda \cup (k,k)}_{\rho_2(n)}(-1) =(-1)^{b(\lambda)+k}\fgr^{\lambda \cup (k,k)}_{\rho_2(n)}(-1)$$
where $\br{\ , \ }$ is the usual scalar product on the ring of symmetric functions.

On the other hand, by \cite[Theorem 2.1 and 2.2]{llt94} we have $Q'_{\lambda \cup (k,k)}(-1) = (-1)^kQ'_{\lambda}(-1)s_k[p_2]$, where $s_k$ is the Schur function corresponding to the partition $(k)$, $p_2$ is the power symmetric function corresponding to the partition $(2)$, and $s_k[p_2]$ is their plethysm. Therefore, $\br{Q'_{\lambda \cup (k,k)}(-1), p_{\rho_2(n)}}$ is also equal to (here we use orthogonality of power symmetric functions with respect to $\br{\ , \ }$)
\begin{align*}
(-1)^k\br{Q'_{\lambda}(-1)s_k[p_2],p_{\rho_2(n)}}&=(-1)^{b(\lambda)+k}\frac{\fgr^\lambda_{\rho_2(n-2k)}(-1)}{z_{\rho_2(n-2k)}}\br{p_{\rho_2(n-2k)}s_k[p_2],p_{\rho_2(n)}}
\\&=(-1)^{b(\lambda)+k}\frac{\fgr^\lambda_{\rho_2(n-2k)}(-1)}{z_{\rho_2(n-2k)}k!}\br{p_{\rho_2(n)},p_{\rho_2(n)}}
\\&=(-1)^{b(\lambda)+k}\frac{\fgr^\lambda_{\rho_2(n-2k)}(-1)}{z_{\rho_2(n-2k)}k!}z_{\rho_2(n)}
\\&=(-1)^{b(\lambda)+k}\frac{m!2^k}{(m-k)!k!}\fgr^\lambda_{\rho_2(n-2k)}(-1).
\end{align*}
Hence the result follows.
\end{proof}
Combining two propositions above, we see that if the main theorem is true for $\twc_\lambda$, then it is also true for $\twc_{\lambda \cup (k,k)}$. Thus by inductive argument, the main theorem is valid if and only if it is valid for strict partitions.

\begin{rmk} We believe that this part is not essential to the proof of the main theorem; it is likely that argument in Section \ref{sec:morph} can be applied to the general cases without assuming that the corresponding partition is strict. However, this assumption is still useful as it simplifies our proof.
\end{rmk}

%\begin{rmk} If $\lambda$ is a Jordan type of some nilpotent element in $\Lie SO_{2m+1}$ or $\Lie Sp_{2m}$, then Proposition \ref{prop:grind} is a corollary of \cite{lus04:ind} together with the fact that $\fgr^\lambda_{\rho_2(n)}(-1)$ is the Euler characteristic of the Springer fiber corresponding to such a nilpotent element. (cf. \cite{kim:total})
%\end{rmk}

\subsection{Asymptotic Hecke algebra and the canonical basis of $K(\cB_N)$} \label{sec:morph}
From now on, we fix a two-sided cell $\twc=\twc_\lambda \subset \extW$ where $\lambda$ is strict. Let $N \in \fg$ be a nilpotent element of Jordan type $\lambda$. By \cite[p.398]{car93}, the reductive part of $Z_G(N)$ (the stabilizer of $N$ in $G$) is a torus isomorphic to $(\bC^\times)^{l(\lambda)}$, which we denote by $F_\twc$. The idea we pursue here is motivated by the conjecture of Lusztig relating the asymptotic Hecke algebra $\jr_\twc$ attached to $\twc$ and the $F_\twc$-equivariant $K$-theory of a certain finite set, see \cite[Conjecture 10.5]{lus89:cell}.

We recall some results of \cite{bez16}. Let $D_{I^0I^0}$ be the category defined in \cite[p.4]{bez16} and $\cP_{I^0I^0}$ be its subcategory of perverse sheaves. Also, define $\tilde{\fg} \colonequals \{ (X, gB) \in \fg \times \cB \mid \Ad(g)^{-1}X \in \fb\}$ equipped with the obvious projection $\tilde{\fg} \rightarrow \fg$. Then we have a natural equivalence of categories \cite[Theorem 1]{bez16}
$$D_{I^0I^0} \simeq D^b(\cP_{I^0I^0}) \simeq D^b(Coh^G_{\cN}(\tilde{\fg}\times_\fg \tilde{\fg})),$$
where $Coh^G_{\cN}(\tilde{\fg}\times_\fg \tilde{\fg})$ is the category of $G$-equivariant coherent sheaves which are set-theoretically supported on the nilpotent cone $\cN\subset \fg$. (Here we identify $G$ with its Langlands dual.) This isomorphism respects the convolution structure on both sides.

Following \cite[11.2]{bez16}, it induces a canonical isomorphism
$$D_{I^0I^0} \simeq D^b(\cA \otimes_{\cO(\fg)}\cA-\md^G_\cN),$$
where $\cA \otimes_{\cO(\fg)}\cA-\md^G_\cN$ is the category of finitely generated $G$-equivariant $(\cA \otimes_{\cO(\fg)}\cA)$-modules that are set-theoretically supported on $\cN \subset \fg$. Here, $\cO(\fg)$ is the coordinate ring of $\fg$ and $\cA$ is a noncommutative Grothendieck resolution of $\fg \times_{\fh/W}\fh$ defined in \cite[1.5]{bemi13}. (See also \cite{bez06}.)

Recall the bijection between two-sided cells and nilpotent orbits in $\fg$ in \cite{lus89:cell}. This bijection is order-preserving \cite{bez09}, and each order induces a filtration on each of two categories above. More precisely, let $D_{I^0I^0,\leq \twc}$ (resp. $D_{I^0I^0,<\twc}$) be the thick subcategory of $D_{I^0I^0}$ generated by irreducible objects $IC_w \in \cP_{I^0I^0}$ for $w \in \twc' \leq \twc$ (resp. $w \in \twc' <\twc$), where $IC_w$ is defined as in \cite[Theorem 55]{bez16}. Then the quotient $D_{I^0I^0,\leq \twc}/D_{I^0I^0,<\twc}$, denoted $D_{I^0I^0,\twc}$, is well-defined. Likewise, for a nilpotent orbit $O \subset \fg$ let $D^b(\cA \otimes_{\cO(\fg)}\cA-\md^G_{\leq O})$ (resp. $D^b(\cA \otimes_{\cO(\fg)}\cA-\md^G_{< O})$) be the full subcategory of $D^b(\cA \otimes_{\cO(\fg)}\cA-\md^G_{\cN})$ consisting of complexes whose cohomology is set-theoretically supported on $\overline{O}$ (resp. $\overline{O} - O$). Then the quotient $D^b(\cA \otimes_{\cO(\fg)}\cA-\md^G_{\leq O})/D^b(\cA \otimes_{\cO(\fg)}\cA-\md^G_{< O})$, denoted $D^b(\cA \otimes_{\cO(\fg)}\cA-\md^G_{O})$, is well-defined.

\cite[Theorem 55]{bez16} states that the isomorphism above respects the filtrations on both sides. In particular, if we set $O$ to be the orbit of $N\in\fg$, then we have a canonical isomorphism
$$D_{I^0I^0,\twc} \simeq D^b(\cA \otimes_{\cO(\fg)}\cA-\md^G_O).$$
Also it sends the perverse $t$-structure on the left to the usual $t$-structure on the right shifted by $a(\twc)=\dim \cB_N$. (See \cite{lus85:cell} for the definition of Lusztig's $a$-function.) Now consider a full subcategory $\I_\twc$ of $D_{I^0I^0,\twc}$ whose objects are (the images of) direct sums of irreducible perverse sheaves and their shifts in $D_{I^0I^0,\leq \twc}$ (under the quotient map). Under the isomorphism above, it is transferred to the full subcategory of $D^b(\cA \otimes_{\cO(\fg)}\cA-\md^G_O)$ whose objects are (the images of) direct sums of irreducible $(\cA \otimes_{\cO(\fg)}\cA)$-modules set-theoretically supported on $\overline{O}$ and their shifts (under the quotient map), which we denote by $\I_O$. Since all the irreducible $(\cA \otimes_{\cO(\fg)}\cA)$-modules set-theoretically supported on $\overline{O}$ are also scheme-theoretically supported on $\overline{O}$, we may also identify $\I_O$ with the full subcategory of $D^b(\cA_N \otimes \cA_N-\md^{Z_G(N)})$ (the bounded derived category of finitely generated $Z_G(N)$-equivalent $(\cA_N \otimes \cA_N)$-modules) whose objects are direct sums of irreducible objects and their shifts. Here, $\cA_N$ is the fiber of $\cA$ at $N \in \fg$.

From this description above, we have canonical isomorphisms
$$K(\I_\twc) \simeq K(D_{I^0I^0,\twc})\simeq K(D^b(\cA \otimes_{\cO(\fg)}\cA-\md^G_O)) \simeq K(\I_O) \simeq K(D^b(\cA_N \otimes \cA_N-\md^{Z_G(N)})).$$
We impose a $\bC$-algebra structure on each term so that they are canonically isomorphic as $\bC$-algebras. First, Lusztig \cite{lus97:cell} defined the truncated convolution on $K(\I_\twc) \simeq K(D_{I^0I^0,\twc})$, which is the usual convolution followed by applying ${}^p\cH^{a(\twc)}$, i.e. taking $a(\twc)$-th perverse cohomology sheaf. Then $K(\I_\twc)$ equipped with this algebra structure is canonically isomorphic to the asymptotic Hecke algebra $\jr_\twc$ attached to $\twc$ (defined over $\bC$). On the other hand, this also induces a truncated convolution on $K(D^b(\cA \otimes_{\cO(\fg)}\cA-\md^G_O)) \simeq K(\I_O)$, which is defined by the usual convolution followed by applying $\cH^0$, i.e. taking $0$-th cohomology sheaf. (This is clear from the comparison of $t$-structures on $D_{I^0I^0,\twc}$ and $D^b(\cA \otimes_{\cO(\fg)}\cA-\md^G_O).$) Therefore, we have a canonical isomorphism of $\bC$-algebras
$$\jr_\twc \simeq K(D^b(\cA_N \otimes \cA_N-\md^{Z_G(N)})).$$
It is also isomorphic to $K^{F_\twc}(\cA_N \otimes \cA_N-\md)$ since $F_\twc$ is the reductive part of $Z_G(N)$. 

There exists a natural morphism (of $\bC$-vector spaces) 
$$K^{F_\twc}(\cA_N \otimes \cA_N-\md) \rightarrow K(\cA_N \otimes \cA_N-\md)$$
which is induced from the forgetful functor. We claim that this morphism is surjective and its kernel is a two-sided ideal of $K^{F_\twc}(\cA_N \otimes \cA_N-\md)$ (with respect to the truncated convolution), thus it induces a $\bC$-algebra structure on $K(\cA_N \otimes \cA_N-\md)$. Indeed, according to \cite[5.2.3]{bemi13}, every irreducible $(\cA_N \otimes \cA_N)$-module can be lifted to an $F_\twc$-equivariant one and such two lifts are isomorphic up to characters of $F_\twc$. (Here we use the assumption that $\lambda$ is strict and thus $F_\twc$ is a torus.) Also, every irreducible $F_\twc$-equivariant $\cA_N \otimes \cA_N$-module arises in this way. From this, the claim easily follows.

On the other hand, inspired by the conjecture of Lusztig \cite{lus89:cell}, Xi \cite{xi02} proved that $\jr_\twc$ is (non-canonically) isomorphic to $Mat_{\X\times \X}(\jr_{\Gamma \cap \Gamma^{-1}})$, where $\Gamma$ is some fixed left cell in $\twc$, $\jr_{\Gamma \cap \Gamma^{-1}}$ is the asymptotic Hecke algebra attached to $\Gamma \cap \Gamma^{-1}$, and $\X=\frac{n!}{\lambda_1!\cdots\lambda_r!}$ is the number of left cells in $\twc$ which is also equal to the Euler characteristic of $\cB_N$. Furthermore, $\jr_{\Gamma \cap \Gamma^{-1}}$ is isomorphic to $Rep(F_\twc)$, which in our case is the $\bC$-algebra of Laurent polynomials in $l(\lambda)$ variables, say $\bC[x_1^{\pm1},x_2^{\pm1},\ldots,x_{l(\lambda)}^{\pm1}]$.

%Then the $\bC$-algebra structure morphism $\bC \rightarrow \jr_{\Gamma \cap \Gamma^{-1}}$ induces an embedding $Mat_{\X \times \X}(\bC) \rightarrow \jr_\twc$.

Let us fix the labeling of the left cells in $\twc$ by $\Gamma_1, \Gamma_2, \ldots, \Gamma_\cX$ once and for all. According to  \cite{xi02}, we may choose an isomorphism $\jr_\twc\simeq Mat_{\X\times \X}(\jr_{\Gamma \cap \Gamma^{-1}})$ such that $(i,j)$-entries in $Mat_{\X\times \X}(\jr_{\Gamma \cap \Gamma^{-1}})$ corresponds to $\Gamma^{-1}_i \cap \Gamma_j$. Now consider the surjection $Mat_{\X\times \X}(\jr_{\Gamma \cap \Gamma^{-1}}) \twoheadrightarrow Mat_{\X\times \X}(\bC)$ induced from the evaluation morphism 
$$\jr_{\Gamma \cap \Gamma^{-1}}\simeq \bC[x_1^{\pm1},x_2^{\pm1},\ldots,x_{l(\lambda)}^{\pm1}] \rightarrow \bC: f(x_1, x_2, \ldots, x_{l(\lambda)}) \mapsto f(1,1, \ldots, 1).$$ Then it is not hard to show that the composition $\jr_\twc \twoheadrightarrow Mat_{\X\times \X}(\bC)$ does not depend on the choice of the isomorphism $\jr_\twc\simeq Mat_{\X\times \X}(\jr_{\Gamma \cap \Gamma^{-1}})$ whenever it respects the fixed labeling of left cells in $\twc$. In other words, there exists a canonical isomorphism $\jr_\twc \twoheadrightarrow Mat_{\X\times \X}(\bC)$ once the order of left cells in $\twc$ is fixed.

So far, we have canonical morphisms
$$
\begin{tikzcd}
\jr_\twc \ar[r, "\simeq"] \ar[d,twoheadrightarrow]&K^{F_\twc}(\cA_N \otimes \cA_N-\md) \ar[d,twoheadrightarrow]
\\Mat_{\X \times \X}(\bC)&K(\cA_N \otimes \cA_N-\md)
\end{tikzcd}
$$
On the other hand, there exists a section map $Mat_{\X \times \X}(\bC)\rightarrow \jr_\twc$ which comes from the $\bC$-algebra structure $\bC \rightarrow \jr_{\Gamma\cap\Gamma^{-1}}$. Composed with $\jr_\twc\simeq K^{F_\twc}(\cA_N \otimes \cA_N-\md) \twoheadrightarrow K(\cA_N \otimes \cA_N-\md)$, it induces a morphism $Mat_{\X \times \X}(\bC) \rightarrow K(\cA_N \otimes \cA_N-\md)$ which makes the above diagram commute. In particular, this morphism is canonical (even though the section map $Mat_{\X \times \X}(\bC)\rightarrow \jr_\twc$ needs not be canonical).

%Therefore, we obtain a chain of canonical homomorphisms of $\bC$-algebras
%$$Mat_{\X \times \X}(\bC) \hookrightarrow \jr_\twc \xrightarrow{\simeq} K^{F_\twc}(\cA_N \otimes \cA_N-\md) \twoheadrightarrow K(\cA_N \otimes \cA_N-\md).$$
Also, we claim that this morphism is an isomorphism. Indeed, since it is a $\bC$-algebra morphism and $Mat_{\X \times \X}(\bC)$ is simple, $Mat_{\X \times \X}(\bC) \rightarrow K(\cA_N \otimes \cA_N-\md)$ is injective. (This map is not zero as it preserves the multiplicative unit.) Now we recall one of the main results in \cite{bemi13}.
\begin{lem} \label{lem:canbasis} There exists a canonical isomorphism $K(\cA_N-\md) \simeq K(\cB_N)$. Under this isomorphism, the basis $\Irr(\cA_N)$ of $K(\cA_N-\md)$ corresponds to the canonical basis of $K(\cB_N)$ defined in \textup{\cite{lus99:kthy2}}.
\end{lem}
In particular, we have $\dim_\bC K(\cA_N \otimes \cA_N-\md) = \X^2 = \dim_\bC Mat_{\X \times \X}(\bC)$,  from which the claim follows. Thus we have a canonical commutative diagram
$$
\begin{tikzcd}
\jr_\twc \ar[r, "\simeq"] \ar[d,twoheadrightarrow]&K^{F_\twc}(\cA_N \otimes \cA_N-\md) \ar[d,twoheadrightarrow]
\\Mat_{\X \times \X}(\bC) \ar[r,"\simeq"]&K(\cA_N \otimes \cA_N-\md)
\end{tikzcd}
$$
Now using the above lemma again, we obtain a canonical isomorphism (of vector spaces)
$$Mat_{\X \times \X}(\bC) \simeq K(\cB_N) \otimes K(\cB_N).$$

\subsection{Involution $\omega$}
We recall the involution $\omega$ on $G$. Clearly, it also induces an involution on $\jr_\twc$ and $Mat_{\X \times \X}(\bC)$. It is clear that there exists a basis  $\{v_{(\Gamma'^{-1}, \Gamma'')} \mid \Gamma', \Gamma'' \textup{ are left cells in } \twc\}$ of $Mat_{\X \times \X} (\bC)$ such that $\omega(v_{(\Gamma'^{-1}, \Gamma'')}) = v_{(\omega(\Gamma'^{-1}), \omega(\Gamma''))}$. Therefore, we have
$$\tr(\omega, Mat_{\X \times \X}(\bC)) = |\{\textup{left cells in } \twc\}^\omega|^2.$$

%On the other hand, $\omega$ also induces an involution on $K(\cB_N)$.
%
%
%From the construction, it is clear that
%$$\tr(\omega, Mat_{\X \times \X}(\bC)) = |\{\textup{left cells in } \twc\}^\omega|^2.$$
On the other hand, if $N\in \fg$ is $\omega$-stable, then $\omega$ also induces an action on $K(\cB_N)$, and by \cite[Lemma 3.2]{hosp77} we have (note that $\rho_2$ is the cycle type of the longest element in $W$)
$$\tr(\omega, K(\cB_N)\otimes K(\cB_N))=\tr(\omega, H^*(\cB_N)\otimes H^*(\cB_N))  = (\fgr^\lambda_{\rho_2(n)}(-1))^2.$$
However, this still makes sense even when $N$ is not $\omega$-stable. Indeed, for any $g \in G$ such that $\Ad(g)(N) = \omega(N)$, we have an isomorphism $\Ad(g)^*: H^*(\cB_{\omega(N)}) \rightarrow H^*(\cB_N)$ which does not depend on the choice of $g$ since $Z_G(N)$ is connected. Also we have a commutative diagram
$$
\begin{tikzcd}[column sep=5em]
H^*(\cB) \ar[r,"\omega"] \ar[d,twoheadrightarrow]& H^*(\cB) \ar[d,twoheadrightarrow] \ar[r,"="]& H^*(\cB) \ar[d,twoheadrightarrow]
\\H^*(\cB_N) \ar[r,"\omega"]& H^*(\cB_{\omega(N)}) \ar[r,"\Ad(g)^*"]&  H^*(\cB_N)
\end{tikzcd}
$$
Thus by identifying $H^*(\cB_N)$ with the quotient of $H^*(\cB)$, the result above is still valid.

Recall that the $\bC$-vector space isomorphism $Mat_{\X \times \X}(\bC) \simeq K(\cB_N) \otimes K(\cB_N)$
is canonical (once the order of left cells in $\twc$ is fixed). As $\omega$ on $Mat_{\X \times \X}(\bC)$ and $K(\cB_N) \otimes K(\cB_N)$ are both induced from the same automorphism $\omega$ on $G$, it follows that this isomorphism is $\omega$-equivariant. \footnote{To be precise, we should check that $\omega$ acts the same way on the Langlands dual of $G$ as on $G$, but it is also true since $\omega$ is self-dual on the root datum of $G$.} In particular, we have
$$|\{\textup{left cells in } \twc\}^\omega|^2=\tr(\omega, Mat_{\X \times \X}(\bC))=\tr(\omega, K(\cB_N)\otimes K(\cB_N)) = ( \fgr^\lambda_{\rho_2(n)}(-1))^2.$$
But since
$$\fgr^\lambda_{\rho_2(n)}(-1)= \tr(\omega, H^*(\cB_N)) = \tr(\omega, K(\cA_N-\md)) =  |\Irr(\cA_N)^\omega|\geq 0,$$
we have $|\{\textup{left cells in } \twc\}^\omega| = \fgr^{\lambda}_{\rho_2(n)}(-1)$. Thus the main theorem is proved.

\begin{rmk} The canonical basis of $K(\cB_N)$ in \cite{lus99:kthy2} is a signed basis, i.e. there is ambiguity on the choice of signs. On the other hand, $\Irr(\cA_N) \subset K(\cA_N-\md)$ is an actual basis, and $\pm\Irr(\cA_N)$ is mapped to Lusztig's canonical basis under the isomorphism $K(\cA_N-\md) \simeq K(\cB_N)$. In our proof, it is crucial that $\omega$ stabilizes not only the signed basis but also $\Irr(\cA_N)$ itself.
\end{rmk}

\section{Proof of Proposition \ref{prop:str}: some combinatorics} \label{sec:comb}
This section is devoted to the proof of Proposition \ref{prop:str}. The argument in this section is explained in \cite{cfkly} in more detail and the proposition also follows from the results therein. However, we still provided its proof here for the sake of completeness.

\subsection{The generalized Robinson-Schensted algorithm}
First, we investigate the connection between left cells in $\extW$ and row-standard Young tableaux under the generalized Robinson-Schensted algorithm originally defined by Shi \cite{shi86, shi91}. Following \cite{lus83:padic}, we identify $\extW$ with the subgroup of $\Aut(\Z)$ defined by 
$$\{w \in \Aut(\Z) \mid \forall i \in \Z, w(n+i) =w(i)+n\}.$$
We express each $w\in \Aut(\Z)$ in terms of the sequence $[w(1), w(2), \ldots, w(n)]$, called the window notation. Then we have
\begin{gather*}
s_i = [1, 2, \ldots, i-1, i+1, i, i+2, \ldots, n],
\\s_0 = [0, 2, \ldots, n-1, n+1], \qquad \tau = [2, 3, \ldots, n, n+1].
\end{gather*}
It is easy to check that they satisfy the defining relations of $\extW$.

For the description of the generalized Robinson-Schensted algorithm, it is natural to consider an infinite version of (standard) Young tableaux. To that end, we define the notion of an infinite periodic sequence as follows. 
\begin{defn} \label{def:piseq}
For each $i \in \Z_{>0}$, let $\tilde{r}_i$ be a finite sequence of integers. Let $\tilde{r}_i(1)$ (resp. $\tilde{r}_i(0)$) be the smallest positive (resp. largest nonpositive) integer in $\tilde{r}_i$ (if exists) and label each element in $\tilde{r}_i$ by its position relative to $\tilde{r}_i(1)$ or $\tilde{r}_i(0)$. Then we call $\tilde{r}=(\tilde{r}_1, \tilde{r}_2, \ldots)$ \emph{an infinite periodic sequence modulo $n$} (\emph{IP sequence mod $n$} for short) if it satisfies the following properties;
\begin{enumerate} 
\item each $\tilde{r}_i$ is strictly increasing, thus in particular $\tilde{r}_i(k)>0$ if and only if $k>0$,
\item $\lim_{i \rightarrow \infty} |\tilde{r}_i| = \infty$, 
\item for any $k \in \Z$, the limit $\lim_{i \rightarrow \infty}(\tilde{r}_i(k))$ exists,
\item there exists $M \in \bN$ such that for any $i>0$ and $\tilde{r}_i=(\tilde{r}_i(s), \tilde{r}_i(s+1), \ldots, \tilde{r}_i(t-1), \tilde{r}_i(t))$, 
$$\tilde{r}_i(k) = \lim_{j\rightarrow \infty} \tilde{r}_j(k) \quad \textup{ for any } s+M\leq k\leq t-M, \quad \textup{ and}$$
\item there exists $1 \leq l \leq n$ and $1\leq a_1<a_2<\cdots<a_l \leq n$ such that 
\begin{gather*}
\lim_{i \rightarrow \infty}\tilde{r}_{i}(kl+r) =kn+a_r \quad \textup{ for any } k \in \Z \textup{ and } 1\leq r \leq l.
\end{gather*}
\end{enumerate}
For any such sequence $\tilde{r}$, it is clear that $l, a_1, \ldots, a_l$ in (5) are uniquely determined if they exist. We define $\Psi_n$ to be the function which sends $\tilde{r}$ to the finite sequence $(a_1, \ldots, a_l)$.
\end{defn}
Likewise, we define an infinite periodic tabloid as follows.
\begin{defn} \label{def:pitab}Let $\tilde{T} = (\tilde{T}_1, \tilde{T}_2, \ldots)$ be an infinite series of Young tabloids such that the following properties hold.
\begin{enumerate} 
\item There exists $M' \in \bN$ such that $\tilde{T}_i$ has $\leq M'$ rows. We define $l(\tilde{T})$ to be the smallest $M'$ which satisfies this property, called the length of $\tilde{T}$.
\item For each $1\leq j \leq l(\tilde{T})$, $\tilde{T}^{(j)} \colonequals \{\tilde{T}_{i}^{(j)}\}_{i\geq 1}$ is an IP sequence mod $n$ where $\tilde{T}_{i}^{(j)}$ is the $j$-th row of $\tilde{T}_i$.
\end{enumerate}
Then we call $\tilde{T}$ \emph{an infinite periodic tabloid modulo $n$} (\emph{IP tabloid mod $n$} for short). For such $\tilde{T}$, we similarly define $\Psi_n(\tilde{T})$ to be the Young tabloid whose rows are $\Psi_n(\tilde{T}^{(1)}),\ldots, \Psi_n(\tilde{T}^{(l(\tilde{T}))})$. Note that an IP sequence mod $n$ is an IP tabloid mod $n$ of length 1. We define $IP_n$ to be the set of infinite periodic tabloids modulo $n$. 
\end{defn}

For any element $w \in \extW$, consider the following sequence that is infinite in both ways:
$$\ldots, w(-3), w(-2), w(-1), w(0), w(1), w(2), w(3), \ldots$$
We consider a sequence $\{(a_i, b_i)\}_{i\geq 1}$ such that $a_i \leq b_i$, $a_i$ is decreasing, $b_i$ is increasing, $\lim_{i\rightarrow \infty}a_i = -\infty$, and $\lim_{i\rightarrow \infty} b_i = \infty$. For each $i$, we consider the standard Young tableaux $\tilde{T}_i$ which is the result of the usual Robinson-Schensted algorithm with input $(w(a_i), w(a_i+1), \ldots, w(b_i-1), w(b_i))$. Then we obtain a series of standard Young tableaux $\tilde{T}=\{\tilde{T}_i\}_{i\geq 1}$. Now we apply the argument in \cite[Section 7]{cpy15} to obtain the following.

\begin{prop} \label{prop:genrs} $\tilde{T} \in IP_n$ and $l(\tilde{T})\leq n$. Also, $\Psi_n(\tilde{T})$ is the same as the result of the generalized Robinson-Schensted algorithm defined in \textup{\cite{shi86, shi91}} applied to $w^{-1}$. (In particular, $\Psi_n(\tilde{T})$ does not depend on the choice of the sequence $\{(a_i, b_i)\}_{i\geq 1}$). Moreover, if $w$ is an element of $W \subset \extW$, then $\Psi_n(\tilde{T})$ is the same as the output of the usual Robinson-Schensted algorithm applied to $w^{-1}$.
\end{prop}
Here $w^{-1}$ appears instead of $w$ since we consider the left action of $\extW$ on $\Z$ instead of the right one. We define $Q(w)$ to be such $\Psi_n(\tilde{T})$ in the theorem and set $P(w) \colonequals Q(w^{-1})$. Then by \cite{shi86,shi91}, $P(w)$ and $Q(w)$ have the same shape. For $w, w' \in \extW$,  $Q(w)=Q(w')$ (resp. $P(w)=P(w')$) if and only if they are contained in the same left (resp. right) cell. Likewise, $Q(w)$ and $Q(w')$ have the same shape if and only if they lie in the same two-sided cell, which is parametrized by the shape of $Q(w)$.

\subsection{Affine Sch{\"u}tzenberger involution and combinatorial $R$-matrix}
Recall the involution $\omega$ acting on $\extW$. Under the identification of $\extW$ with the subset of $\Aut(\Z)$, it corresponds to the conjugation by the element  $k \mapsto 1-k$ in $\Aut(\Z)$. It permutes left cells in $\extW$, thus defines an involution on the set of row-standard Young tableaux under the generalized Robinson-Schensted correspondence. Also, since $\omega$ stabilizes each two-sided cell, it restricts to the involution on $RSYT(\lambda)$ (the set of row-standard Young tableaux of shape $\lambda$) for each $\lambda \vdash n$. 

We claim that this action can be described in terms of combinatorial $R$-matrices. (This is originally proved by Chmutov-Lewis-Pylyavskyy.) First we recall the definition of combinatorial $R$-matrices on the tensor product of single-row Kirillov-Reshetikhin crystals (KR crystals for short). We refer readers to \cite{shi05} for a nice exposition on this subject. For $a, b \in \Z_{>0}$, we regard $RSYT((a,b))$ as a subset of the vertices of the crystal graph $\bB^b \otimes \bB^a$, where $\bB^s$ is the KR crystal of shape $(s)$ (of $U_q(\widehat{\mathfrak{sl}_k})$ for a suitable choice of $k$). Then there exists a unique isomorphism $\bB^b \otimes \bB^a \rightarrow \bB^a \otimes \bB^b$, which we call the combinatorial $R$-matrix, and it restricts to a bijection
$$\Rmat: RSYT((a,b)) \rightarrow RSYT((b,a)).$$
\cite[Example 4.10]{shi05} describes this operation using jeu-de-taquin and sliding process. Here we briefly explain his description with an example.

\begin{example} \label{ex:rmat}
\ytableausetup{smalltableaux}
Let $a=4, b=3$ and $T = \begin{ytableau}
3 &4 & 6 & 7 \\
1 & 2 & 5 \\
\end{ytableau}  \in RSYT((a,b))$. To apply $\Rmat$, we first draw the skew-shaped standard Young tableau $\begin{ytableau}
\none & \none& \none&3 &4 & 6 & 7 \\
1 & 2 & 5 \\
\end{ytableau}$
obtained from sliding the first row to the right, and apply jeu-de-taquin process until each row has the correct number of boxes. 
$$\begin{ytableau}
\none & \none& \none&3 &4 & 6 & 7 \\
1 & 2 & 5 & \bullet \\
\end{ytableau} \rightarrow \begin{ytableau}
 \none& \none&3 &4 & 6 & 7 \\
1 & 2 & 5 & \bullet \\
\end{ytableau}\rightarrow \begin{ytableau}
\none&\none&\none&4 & 6 & 7 \\
1 & 2 & 3 & 5 \\
\end{ytableau}$$
As a result, we have $\Rmat(T) = \begin{ytableau}
4 & 6 & 7 \\
1 & 2 & 3 & 5 \\
\end{ytableau}$. Note that both $\begin{ytableau}
\none & \none& \none&3 &4 & 6 & 7 \\
1 & 2 & 5 \\
\end{ytableau}$ and $\begin{ytableau}
\none & \none& \none&\none &4 &6  & 7 \\
1 & 2 & 3&5 \\
\end{ytableau}$ are jeu-de-taquin equivalent to the standard Young tableau $\begin{ytableau}
1& 2& 3&4 &6  & 7 \\
  5 \\
\end{ytableau}$. In general, the combinatorial $R$-matrix does not change the associated jeu-de-taquin equivalent standard Young tableau.

Or, first we again consider $\begin{ytableau}
\none & \none& \none&3 &4 & 6 & 7 \\
1 & 2 & 5 \\
\end{ytableau}$ and slide each box on the second row to the rightest with preserving semi-standard property to get $\begin{ytableau}
\none & \none&3 &4 & 6 & 7 \\
1 & 2 & \none &5 \\
\end{ytableau}$. Then, push down the correct number of leftmost boxes (in this case we push down $4-3=1$ box) from the first row to obtain $\begin{ytableau}
\none & \none&\none &4 & 6 & 7 \\
1 & 2 & 3 &5 \\
\end{ytableau}$. Thus we also see that $\Rmat(T) = \begin{ytableau}
4 & 6 & 7 \\
1 & 2 & 3 & 5 \\
\end{ytableau}$.

If $a<b$, then we first slide each box in the first row to the leftest with preserving semi-standard property and push up the correct number of rightmost boxes from the second row. 
\end{example}

This combinatorial $R$-matrix is generalized to any finite tensor product of single-row KR crystals. In particular, for any sequence of positive integers $\lambda=(\lambda_1, \ldots, \lambda_r)$ (not necessarily a partition) we similarly define
$$\Rmat_i : RSYT(\lambda) \rightarrow RSYT((\lambda_1, \ldots, \lambda_{i-1}, \lambda_{i+1}, \lambda_i, \lambda_{i+2}, \ldots, \lambda_r))$$
to be the corresponding combinatorial $R$-matrix. From the theory of crystals, we easily deduce the following properties of $\Rmat$.
\begin{enumerate}
\item $\Rmat_i^2 = Id$.
\item $\Rmat_i\Rmat_{i+1}\Rmat_i = \Rmat_{i+1}\Rmat_i\Rmat_{i+1}$.
\item If $\lambda_i = \lambda_{i+1}$, then $\Rmat_i = Id$.
\item In general, if $\lambda_1, \ldots, \lambda_r, \mu_1, \ldots, \mu_r\in \bN$ such that $\{\lambda_1, \ldots, \lambda_r \} = \{\mu_1, \ldots, \mu_r \}, $ then all the compositions of combinatorial $R$-matrices from $RSYT((\lambda_1, \ldots, \lambda_r))$ to $RSYT((\mu_1, \ldots, \mu_r))$ give the same map.
\end{enumerate}

As promised, we illustrate $\omega$ in terms of combinatorial $R$-matrices as follows.
\begin{prop} \label{prop:rmat}Let $\lambda \vdash n$ be a partition. For a given $T \in RSYT(\lambda)$, $T \mapsto \omega(T)$ is equivalent to the following process:
\begin{enumerate}
\item rotate $T$ by $180^\circ$ and push each row to the left so that it becomes a Young tabloid, 
\item substitute each entry $i$ by $n+1-i$ to make it row-standard, and
\item apply combinatorial $R$-matrices accordingly to retain the original shape $\lambda$.
\end{enumerate}
\end{prop}

First, we restrict our attention to $W \subset \extW$ and $SYT(\lambda) \subset RSYT(\lambda)$ (the set of standard Young tableaux of shape $\lambda$). Then for any $w \in W$, $\omega(w) = w_0ww_0$ where $w_0 \in W$ is the longest element in $W$. Under the Robinson-Schensted algorithm, this corresponds to the usual Sch{\"u}tzenberger involution. Also, it follows from \cite[Appendix A]{sta86} that this involution is the same as the one described in Proposition \ref{prop:rmat}. Therefore, this proposition is true for elements in $SYT(\lambda)$.

In general, let $\tilde{T}$ be an IP tabloid mod $n$ such that $\Psi_n(\tilde{T})$ is a row-standard Young tabloid. We claim that combinatorial $R$-matrices and the function $\Psi_n$ behave well together as follows.
\begin{lem} \label{lem:comm}For $1 \leq i \leq l(\tilde{T})-1$, let $\Rmat_i(\tilde{T})$ be the series $(\Rmat_i(\tilde{T}_1), \Rmat_i(\tilde{T}_2), \ldots)$. Then $\Rmat_i(\tilde{T})$ is again an IP tabloid mod $n$ and we have $\Psi_n(\Rmat_i(\tilde{T}))= \Rmat_i(\Psi_n(\tilde{T})).$
\end{lem}
\begin{proof}It suffices to assume that $l(\tilde{T})=2$ and $i=1$. In this case, it is an easy combinatorial exercise using the description of $\Rmat$ in terms of sliding process in \cite[Example 4.10]{shi05}.
\end{proof}
\begin{example} Suppose $\tilde{T}$ is an IP tabloid mod 7 such that $\Psi_7(\tilde{T}) = \begin{ytableau}
3 &4 & 6 & 7 \\
1 & 2 &5 \\
\end{ytableau}$. Then (the limit of) each row of $\tilde{T}$ looks like
\ytableausetup{nosmalltableaux}
\ytableausetup{boxsize=1.7em}
\begin{gather*}
\begin{ytableau}
\cdots &-11 &-10 & -8 & -7&-4 &-3 & -1 & 0& 3 &4 & 6 & 7 &10 &11 & 13 & 14&17 &18 & 20 & 21 & \cdots
\end{ytableau}, \quad \textup{ and}
\\\begin{ytableau}
\cdots&-13 &-12 & -9 &-6 &-5 & -2 & 1& 2 &5 & 8 & 9 &12 &15 & 16 &19 &\cdots
\end{ytableau}.
\end{gather*}
Now we put each box in the second row to the rightest with keeping semi-standard property. Then it looks like
\begin{gather*}
\begin{ytableau}
\none&\none&\cdots &-11 &-10 & -8 & -7&-4 &-3 & -1 & 0& 3 &4 & 6 & 7 &10 &11 & 13 & 14&17 &18 & 20 & 21 & \cdots
\\
\cdots&-13 &-12 &\none& -9 &-6 &-5 &\none& -2 & 1& 2 &\none& 5 & 8 & 9 &\none &12 &15 & 16 &\none &19 &\cdots
\end{ytableau}
\end{gather*}
\ytableausetup{smalltableaux}
But this is a concatenation of $\begin{ytableau}
\none&\none& 3 &4 & 6 & 7 
\\1& 2 &\none& 5
\end{ytableau}$ and its shifts by a multiple of $7$. Note that $\begin{ytableau}
\none&\none& 3 &4 & 6 & 7 
\\1& 2 &\none& 5
\end{ytableau}$ appears in the usual combinatorial $R$-matrix calculation on $\Psi_7(\tilde{T}) = \begin{ytableau}
3 &4 & 6 & 7 \\
1 & 2 &5 \\
\end{ytableau}$ (cf. Example \ref{ex:rmat}). Also, if one pushes down all the possible boxes from the first row, then it corresponds to pushing down the box of entry 3 in $\begin{ytableau}
\none&\none& 3 &4 & 6 & 7 
\\1& 2 &\none& 5
\end{ytableau}$. Therefore, in this case we get $\Psi_7(\Rmat(\tilde{T}))=\Rmat(\Psi_7(\tilde{T}))$.

Indeed, it is not hard to show that the image of $\Psi_n\circ \Rmat$ only depends on $\Psi_n(\tilde{T})$; if $\tilde{T}, \tilde{T}'$ are two IP tabloids mod $n$ such that $\Psi_n(\tilde{T}) = \Psi_n(\tilde{T}')$, then indeed we have $\Psi_n(\Rmat(\tilde{T}))=\Psi_n(\Rmat(\tilde{T}'))$. In other words, one may simply ignore ``finite error" part in both ends of IP sequences mod $n$ because of the condition (4) in Definition \ref{def:piseq}.
\end{example}

\begin{proof}[Proof of Proposition \ref{prop:rmat}] Suppose $w \in \extW$ is given and $\tilde{T}=(\tilde{T}_1, \tilde{T}_2, \ldots)$ is an IP tabloid mod $n$ constructed in Proposition \ref{prop:genrs}. Here we choose $(a_i, b_i)_{i\geq 1}$ such that $a_i+b_i=1$. (This assumption is not necessary but it simplifies the proof.) In other words, each $\tilde{T}_i$ is the output of the usual Robinson-Schensted algorithm applied to the sequence $w(a_i), w(a_i+1), \ldots, w(b_i-1), w(b_i)$. If we apply $\omega$, then it corresponds to replacing the sequence $w(a_i), w(a_i+1), \ldots, w(b_i-1), w(b_i)$ with $\omega(w)(a_i), \omega(w)(a_i+1), \ldots, \omega(w)(b_i-1), \omega(w)(b_i), $ i.e.
$$1-w(1-a_i),1-w(-a_i), \ldots, 1-w(2-b_i), 1-w(1-b_i)$$
which is equal to
$$1-w(b_i),1-w(b_i-1), \ldots, 1-w(1+a_i), 1-w(a_i)$$
since $a_i+b_i=1$. According to \cite[Appendix A]{sta86}, this is similar to the usual Sch{\"u}tzenberger involution. Indeed, the output of usual Robinson-Schensted algorithm applied to $1-w(b_i),1-w(b_i-1), \ldots, 1-w(1+a_i), 1-w(a_i)$ can also be obtained from $\tilde{T}_i$ by the following process.
\begin{enumerate}
\item Rotate $\tilde{T}_i$ by $180^\circ$ and push each row to the left so that it becomes a Young tabloid, 
\item substitute each entry $i$ by $1-i$, and
\item find a standard Young tableau which is jeu-de-taquin equivalent to the result of (2) with the same shape as $\tilde{T}_i$.
\end{enumerate}
From the properties of combinatorial $R$-matrices, step (3) is also equivalent to the following.
\begin{enumerate}[label=(3')]
\item apply combinatorial $R$-matrices accordingly to retain the original shape of $\tilde{T}_i$.
\end{enumerate}
Now we apply $\Psi_n$ on the output of each $\tilde{T}_i$ under this process. Step (1) obviously commutes with $\Psi_n$, and so does step (2) modulo $n$. Also, step (3') commutes with $\Psi_n$ by Lemma \ref{lem:comm}. Therefore, $\omega(\Psi_n(\tilde{T}))$ is obtained from $T=\Psi_n(\tilde{T})$ by applying the process below:
\begin{enumerate}
\item Rotate $T$ by $180^\circ$ and push each row to the left so that it becomes a Young tabloid, 
\item substitute each entry $i$ by $1-i$ modulo $n$, say $n+1-i$, and
\item apply combinatorial $R$-matrices accordingly to retain the original shape of $T$.
\end{enumerate}
But this is what we want to prove.
\end{proof}

The description of $\omega$ in terms of combinatorial $R$-matrices has an advantage that it can be generalized to any row-standard Young tabloid, say $RSYT(\lambda)$ where $\lambda$ is a finite sequence of positive integers which is not necessarily a partition, since the method described in Proposition \ref{prop:rmat} does not rely on the condition that $\lambda$ is a partition. If we write such a generalization again by $\omega$, then the following lemma is easily proved.
\begin{lem}\label{lem:comm2} Let $\lambda = (\lambda_1, \ldots, \lambda_r)$ be a finite sequence of positive integers. Then for $1\leq i \leq r-1$, the two maps 
$$\omega \circ \Rmat_i, \Rmat_{i}\circ \omega\colon RSYT(\lambda) \rightarrow RSYT((\lambda_1, \ldots, \lambda_{i-1}, \lambda_{i+1}, \lambda_{i}, \lambda_{i+2}, \ldots, \lambda_r))$$
coincide. In other words, $\Rmat_i$ and $\omega$ ``commute".
\end{lem}

\subsection{Proof of Proposition \ref{prop:str}} By Proposition \ref{prop:rmat} and Lemma \ref{lem:comm2}, we may illustrate $T \mapsto \omega(T)$ for $T\in RSYT(\lambda\cup(k,k))$ by the following process.
\begin{enumerate}
\item Apply $\Rmat$ accordingly to obtain a row-standard Young tabloid of shape $(k, \lambda_1, \ldots, \lambda_r, k)$.
\item Apply (the generalized version of) $\omega$ which is an involution on $RSYT((k, \lambda_1, \ldots, \lambda_r, k))$.
\item Apply $\Rmat$ accordingly to obtain a row-standard Young tableau of shape $\lambda \cup (k,k)$.
\end{enumerate}
Note that step (1) and (3) are inverse to each other. Therefore by Lemma \ref{lem:comm2}, we have
$$|RSYT(\lambda\cup(k,k))^\omega| = |RSYT((k, \lambda_1, \ldots, \lambda_r, k))^\omega|.$$
Now, we consider the surjection 
$$\Phi\colon RSYT((k, \lambda_1, \ldots, \lambda_r, k)) \rightarrow RSYT(\lambda)$$
which sends $T= (T^{(1)}, T^{(2)}, \ldots, T^{(r)}, T^{(r+1)}, T^{(r+2)})$ to the renormalization of $(T^{(2)}, \ldots, T^{(r)}, T^{(r+1)})$, i.e. removes the first and the last row (of length $k$) of $T$ and renormalizes the result so that the entries are $1, 2, \ldots, |\lambda|$. For example, 
\ytableausetup{smalltableaux}
$$\textup{if } k=2, \lambda=(4,3), \textup{ and } T=\begin{ytableau}
3&9\\
1 & 2 &7&11\\
5 &6&10\\
4 &8\\
\end{ytableau}, \textup{ then } \Phi(T)= \begin{ytableau}
1 & 2 &5&7\\
3 &4&6\\
\end{ytableau}.
$$
Now, from the description of $\omega$ (Proposition \ref{prop:rmat}), $T$ is $\omega$-stable if and only if
\begin{enumerate}
\item there exists $a_1,a_2, \ldots, a_k \in \{1, 2, \ldots, n\}$ such that $a_1<a_2<\cdots<a_k$ and $\{a_1, \ldots, a_k\} \cap \{n-a_1, \ldots, n-a_k\}=\emptyset$, and the first (resp. last) row of $T$ is $(a_1, a_2, \ldots, a_k)$ (resp. $(n-a_k, n-a_{k-1}, \ldots, n-a_1)$),
\item $\Phi(T)$ is $\omega$-stable.
\end{enumerate}
Thus, $\varphi(\lambda \cup(k,k))$ equals $\varphi(\lambda)$ multiplied by the number of choices of such $\{a_1, a_2, \ldots, a_k\}$, which is equal to $\binom{m}{k}2^k$ where $m = \floor{n/2}$. But this is what we want to prove.

\bibliographystyle{amsalphacopy}
\bibliography{leftcell}

\end{document}